\documentclass[amscd,amssymb,verbatim,12pt]{amsart}
\usepackage{epsfig}
\usepackage{graphicx}
\newif\ifAMS
\IfFileExists{amssymb.sty}
  {\AMStrue\usepackage{amssymb}}
  {\usepackage{latexsym}}
  \usepackage{enumerate}
\theoremstyle{plain}

\newtheorem{Thm}{Theorem}[section]
\newtheorem{Cor}[Thm]{Corollary}
\newtheorem{Lem}[Thm]{Lemma}

\theoremstyle{definition}
\newtheorem{Def}{Definition}

\theoremstyle{remark}

\newtheorem{Rem}{Remark}

\newcommand{\interior}{^{ \kern-5pt ^\circ}}
\newcommand {\bd}{\partial}

\begin{document}
\title
{Codimension one subgroups and boundaries of hyperbolic groups}

\author
{Thomas Delzant and Panos Papasoglu  }

\subjclass{  }

\email [Thomas Delzant]{delzant@math.u-strasbg.fr} \email [Panos
Papasoglu]{panos@math.uoa.gr}

\address
[Thomas Delzant] {Institut de Recherche Math\'ematique Avanc\'ee,
Universit\'e Louis Pasteur et CNRS, 7 rue Ren\'e Descartes, 67084
Strasbourg Cedex, France  }
\address
[Panos Papasoglu] {Mathematics Department, University of Athens,
Athens 157 84, Greece }
\thanks {We acknowledge support from the French-Greek grant Plato}
\begin{abstract}
We construct hyperbolic groups with the following properties: The
boundary of the group has big dimension, it is separated by a
Cantor set and the group does not split. This shows that
Bowditch's theorem that characterizes splittings of hyperbolic
groups over 2-ended groups in terms of the boundary can not be
extended to splittings over more complicated subgroups.

\end{abstract}
\maketitle
\section{Introduction}

Let $G$ be a finitely generated group and let $H$ be a subgroup of
$G$. We say that $H$ is a co-dimension 1 subgroup if $C_G/H$ has
more than 1 end, where $C_G$ is the Cayley graph of $G$. If $G$
splits over $H$ then one easily sees that $H$ is co-dimension 1.
The opposite is not true, for example any closed geodesic on a
surface group gives a cyclic codimension 1 subgroup of the
fundamental group of the surface. On the other hand only simple
closed geodesics correspond to splittings.

The surface example can be generalized to $CAT(0)$ complexes to
produce examples of codimension 1 subgroups: If $X$ is a finite
$CAT(0)$ complex of (say) dimension 2 and if $R$ is a locally
geodesic track on $X$ then the subgroup of $G=\pi _1 (X)$
corresponding to $R$ is a codimension 1 free subgroup of $G$. Wise
(\cite {W}) has exploited this idea producing codimension 1
subgroups for small cancellation groups. In the setting of small
cancellation groups of course one needs some combinatorial analog
for the convexity property of geodesics (or tracks) and Wise
develops such a notion. Pride (\cite {Pr}) has shown that there
are small cancellation groups that have property FA, so such
groups have codimension 1 subgroups but do not split.

Stallings showed that if a compact set separates the Cayley graph
of a finitely generated group $G$, into at least two unbounded
components, then $G$ splits over a finite group. Bowditch (\cite
{B}) showed something similar for hyperbolic groups: if the
boundary $\bd G$ of a 1-ended hyperbolic group $G$ has a local cut
point, then the group splits over a 2-ended group, unless it is a
triangle group. There have been other generalizations of Stallings
theorem similar in spirit. The general idea is that if a `small'
set (coarsely) separates the Cayley graph of a group then the
group splits over a subgroup quasi-isometric to the `small set'.
For a precise conjecture see \cite{P}.

The main purpose of this paper is to show the limitations of this
`philosophy'. Given any $n>0$, we produce an example of a
hyperbolic group $G$, such that $dim(\bd G)>n$, $\bd G$ is
separated by a set of dimension 0 (a Cantor set) and $G$ has
property $FA$ (so it does not split over any subgroup). Our
example is based on Wise's construction which we generalize to the
setting of small cancellation theory over free products.

\section{Preliminaries}

\begin{Def} A diagram is a finite connected planar graph.
The faces of a diagram $D$ are the closures of the bounded
components of $\Bbb R ^2-D$.
\end{Def}
In what follows we assume always that each interior vertex (i.e.
not on $\bd D$) of a diagram has degree at least 3. We can always
achieve this by erasing all interior vertices of degree 2.

We will need some small cancellation results about diagrams shown
by McCammond and Wise in \cite{MW}. For the reader's convenience
and also because our setting is slightly different we include
these results here. These results strengthen classical small
cancellation results (see e.g. \cite{L-S}).

We need some notation: If $D$ is a diagram we denote by $\bd D$
the boundary of the unbounded component of $\Bbb R ^2-D$ (so if
$U$ is the unbounded component of $\Bbb R ^2-D$, $\bd D=\bar
U-U$). We say that the diagram is \textit{non singular} if $\bd D$
is homeomorphic to $S^1$. We say that an edge of $D$ is interior
edge if it does not lie on $\bd D$.

If $D$ is a diagram we denote by $E,F,V$ respectively the total
number of edges, faces and vertices of the diagram.

We denote by $E^\bullet , E^\circ $ respectively the number of
edges of the diagram that lie (do not lie) on $\bd D$. We denote
by $V^+$ the number of vertices on $\bd D$ that lie in exactly one
face and by $V^-$ the number of vertices on $\bd D$ that lie on
more than one face. We denote by $V^\circ $ the number of vertices
of $D$ that do not lie on $\bd D$.   We say that a diagram
verifies the $C(6)$ condition if every face of the diagram has at
least 6 sides. We have the following version of Greedlinger's
lemma (see \cite{L-S}):

\begin{Lem}\label{Greed} Let $D$ be a non singular diagram which verifies the
condition $C(6)$. Then $V^+\geq V^-+6$.
\end{Lem}
\begin{proof}
We have the following inequalities:

$$ 6F\leq 2E^\circ + E^\bullet $$
This is because each face has at least 6 edges and each interior
edge lies in at most 2 faces while boundary edges lie in one face.

$$ 2E\geq 3V^\circ+3V^-+2V^+$$

This is because each edge has at most 2 vertices and each interior
edge has degree at least 3.

Using Euler's formula and the first inequality we obtain:

$$V=E-F+1 \geq E-\frac {E^\circ}{3}-\frac {E^\bullet}{6}+1=\frac
{2E}{3}+\frac {E^\bullet}{6}+1$$

We remark now that $$E^\bullet = V^-+V^+$$ Substituting $E^\bullet
$ above and using the second inequality for $E$ we obtain

$$V=V^-+V^++V^\circ \geq \frac {2}{3}(\frac {3}{2}V^\circ+\frac
{3}{2}V^-+V^+)+\frac {V^-+V^+}{6}+1\Rightarrow $$
$$ V^+\geq V^-+6$$
\end{proof}
We recall some definitions from \cite{MW}:
\begin{Def} Let $D$ be a non singular diagram. A face $F$ of the
diagram is called an $i$-spur if the intersection of $F$ with the
boundary of $D$ is connected and exactly $i$-edges of $F$ are
interior edges of $D$.
\end{Def}
\begin{Def}
We say that a diagram $D$ is a ladder if there are at most two
faces $F_1,F_2$ of $D$ such that $D-F_1,D-F_2$ are connected while
for every other face $F$ of $D$, $D-F$ has exactly 2 components.
\end{Def}
\begin{figure}[h]
\includegraphics[width=4.2in ]{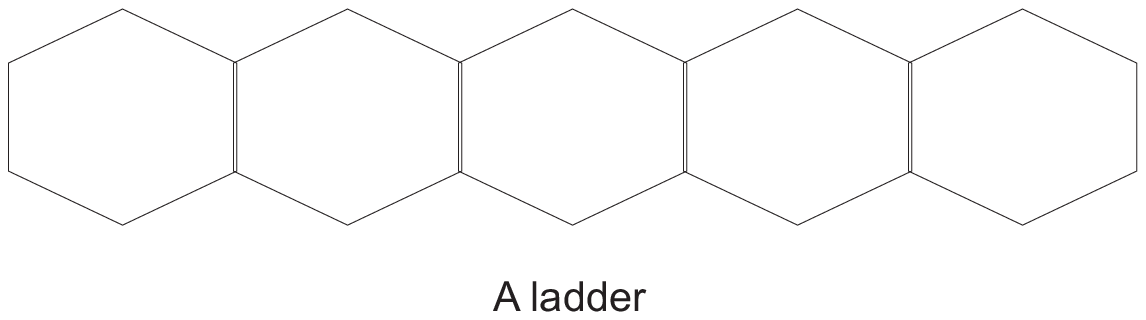}
\end{figure}
We have the following corollary from lemma \ref{Greed}:
\begin{Cor}\label{ladder} Let $D$ be a non singular disc diagram which is
$C(6)$ and contains no $3$-spurs and at most two $i$-spurs for
$i\leq 2$. Then either $D$ has a single region or it contains
exactly two $i$-spurs with $i\leq 2$ and it is a ladder.
\end{Cor}
\begin{proof}
We modify $D$ as follows: if a face $F$ of $D$ has more than 6
edges and it intersects the boundary we erase successively
vertices of $F$ that do not lie on any other face till $F$ has 6
edges (or there are no more vertices to erase). Let's call $D_1$
the new diagram. $D_1$ is still a $C(6)$ diagram. $D_1$ contains
also no $3$-spurs and at most two $i$-spurs for $i\leq 2$. We
consider now a face $F$ of $D_1$ that intersects the boundary and
we see how it contributes to $V^+,V^-$. If $F$ is not an $i$-spur
then 2 vertices of $F$ contribute to $V^+$ while at least 4
vertices of $F$ contribute to $V^-$. So the total contribution of
all such faces to the difference $V^+-V^-$ is at most 0 (note that
the contribution is not necessarily negative as we count twice the
$V^-$ vertices as they lie in at least 2 faces). The contribution
of an $i$-spur to the difference $V^+-V^-$ is $4-i$.

 Since $D$
contains no 3-spurs and at most 2 $i$-spurs for $i\leq 2$, the
inequality
$$V^+- V^-\geq 6$$ implies that if $D_1$ has more than one face
then $D_1$ has exactly two 1-spurs, say $F_1,F_2$. If we erase
$F_1$ we obtain a diagram $D_2$ which is again $C(6)$. We note
that $F_1$ intersects exactly one face of $D_1$ so after erasing
it the diagram $D_2$ has still the other 1-spur of $D_1$ and at
most one new $i$-spur for some $i\leq 3$. By the inequality
$V^+\geq V^-+6$ again we conclude as before that either $D_2$ has
only one face or it has exactly two 1-spurs $F_2,F_3$. Inductively
we see that $D_1$ is a ladder hence $D$ is also a ladder.
\end{proof}

We will need a more technical result. If $v$ is a vertex in a
diagram we denote by $d_v$ the valency of $v$. The result below
will be used to show that small cancellation products of word
hyperbolic groups are word hyperbolic.
\begin{Lem}\label{isop} Let $D$ be a non singular diagram that verifies the
condition $C(7)$. Then
$$\frac {1}{3}\sum _{v\in D^{0}}\frac{d_v}{2}-\frac {2E^\circ
}{7}\leq V^\bullet +\frac {E^\bullet}{7}$$

In particular
$$F\leq 3E^\bullet+ 3V^\bullet$$ i.e. $D$ satisfies a
linear isoperimetric inequality.
\end{Lem}
\begin{proof}

We denote by $D^0$ the set of vertices of $D$ (the 0-skeleton).

Clearly
$$\sum _{v\in D^{0}}\frac{d_v}{2}=E\ \ \ \ \ \ \ (1)$$

 We have also the following inequality:

$$ 7F\leq 2E^\circ + E^\bullet $$
This is because each face has at least 7 edges and each interior
edge lies in at most 2 faces while boundary edges lie in one face.

Using Euler's formula and the inequality above we obtain:

$$E+1=V+F\leq  V^\bullet+\frac {E^\bullet}{7}+V^\circ+\frac {2E^\circ}{7}\ \ \ \ \ (2)$$
Since $d_v\geq 3$ for every $v$ in the interior of $D$:

$$\sum _{v\in D^{0}}\frac{d_v}{2}-V^\circ\geq \frac {1}{3}\sum _{v\in
D^{0}}\frac{d_v}{2} $$ By (1) and (2) we have

$$\frac {1}{3}\sum _{v\in D^{0}}\frac{d_v}{2}-\frac {2E^\circ
}{7}\leq V^\bullet +\frac {E^\bullet}{7}\ \ \ \ \ (3)$$

Since
$$\frac {1}{3}\sum _{v\in D^{0}}\frac{d_v}{2} \geq \frac {E^\circ
}{3} $$ we have

$$\frac {1}{3}\sum _{v\in D^{0}}\frac{d_v}{2}-\frac {2E^\circ
}{7} \geq \frac {E^\circ }{42} \geq \frac {3F}{7}-\frac {E^\bullet
}{84} $$ and using (3)

$$V^\bullet +\frac {2E^\bullet}{7}\geq  \frac {3F}{7}-\frac {E^\bullet
}{84}\Rightarrow F\leq 3E^\bullet+ 3V^\bullet$$

\end{proof}

\section{Small cancellation theory over free products}

Small cancellation theory can be developed over free products (see
\cite {L-S}). We show in this section that small cancellation
products have codimension 1 subgroups. This generalizes a result
of Wise (\cite {W}). We recall that the free product factors embed
in small cancellation products (\cite {L-S} cor. 9.4, p.278). Osin
(\cite{Os}, lemma 4.4) showed that free product factors embed
quasi-isometrically in small cancellation products (this also
follows from \cite{De}). For the reader's convenience we include a
proof of this below.

\begin{Def} Let $<S|R>$ be a presentation of a group $G$. We say
that $<S|R>$ is symmetrized if for any $r=y_1...y_n\in R$ all $n$
cyclic permutations of $r$ are also in $R$ and $r^{-1}$ is in $R$
too. We assume that all elements of $R$ are reduced words. If
$r_1=cb, r_2=ca$ and the words $cb,ca$ are reduced we call $c$ a
\textit{piece} of the presentation.
\end{Def}

Let now $<S|R>$ be a symmetrized presentation. We have then the
following small cancellation conditions:
\begin{itemize}
\item \textit{Condition} $C'(\lambda )$: If $r\in R$ and $r=cb$
with $cb$ reduced word and $c$ a piece then $|c|<\lambda |r|$.

\item \textit{Condition} $C(p)$: No element of $R$ is a product of
fewer than $p$ pieces.

 \item \textit{Condition} $B(2p)$: If $r=ab$ and $a$ is a product
 of $p$ pieces then $|a|\leq |r|/2$.
\end{itemize}

Wise showed in \cite{W} that groups that admit a presentation in
which all relators have even length and condition $B(6)$ is
satisfied, have codimension 1 subgroups. Clearly condition $C'(1/6
)$ is stronger than condition $B(6)$ so Wise's result holds for
these groups too.

We recall now how that the small cancellation conditions can be
given for free products too (\cite {L-S} ch V, sec. 9). Let $G$ be
the free product of the groups $A_i$.

We say that a word $a_1...a_n$ is reduced if each $a_j$ represents
an element of one of the $A_i$ and $a_j,a_{j+1}$ belong to
different factors for any $j$. Any element $g\in G$ can be
represented in a unique way as a reduced word (normal form of
$g$). If $g=a_1...a_n$ is the normal form of $g$
 we define $\|g\|=n$. If
$u=a_1...a_n,v=b_1...b_k$ are reduced words we say that the word
$uv=a_1...a_nb_1...b_k$ is \textit{semi-reduced} if $a_nb_1\ne e$.
Note however that $a_n,b_1$ might lie in the same factor. We say
that a word $w=a_1...a_n$ is \textit{weakly cyclically reduced} if
it is reduced and $a_n^{-1}a_1\ne e$. We say that a sequence of
words $R$ is symmetrized if whenever $r\in R$ all weakly
cyclically reduced conjugates of $r$ and $r^{-1}$ are in $R$. We
say that $c$ is a \textit{piece} if there are distinct $r_1,
r_2\in R$ such that $r_1=ca,r_2=cb$ and the words $ca,cb$ are
semi-reduced. As before we have the condition $C'(\lambda )$:

\textit{Condition} $C'(\lambda )$: If $r\in R$ and $r=cb$ with
$cb$ semi-reduced word and $c$ a piece then $\|c\|<\lambda \|r\|$.

Let now $F$ be a free product $F=*A_i$ and let $R$ be a
symmetrized subset of $F$. The group $G$ defined by the free
product presentation $<F|R>$ is the quotient $$G=F/<<R>>$$ where
$<<R>>$ is the normal closure of $R$ in $F$.

 We show now that if a group $G$ has a free product presentation
  $<F|R>$ that satisfies the $C'(1/6)$ condition then
$G$ has a codimension 1 subgroup. We start first by considering
van-Kampen diagrams over $G$. We consider a usual presentation of
$G$ with a set of generators $S$ given by the generators of
$A_i's$ and a set of relators consisting of relators of the
$A_i'$s together with a set $R'$ such that $R$ is obtained by
taking all weak cyclic conjugates of elements in $R'$ and their
inverses. If $R'$ is finite we say that $G$ has a \textit{finite
free product presentation.} Let now $w$ be a word in $S$
representing the identity in $G$. Let $D$ be a reduced Van-Kampen
diagram for $w$ for the presentation given above. We remark that
if $p$ is a simple closed path in the 1-skeleton of $D$ such that
all edges of $p$ lie in a single factor $A_i$ then the word
corresponding to $p$ represents the identity in $A_i$ (see
\cite{L-S}, cor. 9.4). Call such a simple closed path maximal if
there is no other such simple closed path $q$ in the interior of
$p$. We modify now the diagram $D$ as follows: For each maximal
simple closed path $p$ we erase all edges of $p$ and all edges of
$D$ inside $p$ and we introduce a new vertex $v_p$ which we join
with all vertices of $p$. Now each edge $e$ of $p$ has been
replaced by two edges $e_1,e_2$. We label $e_1,e_2$ by elements of
$A_i$ so that the product of their labels is equal to the label of
$e$. This is clearly possible since $p$ corresponds to the trivial
element in $A_i$. We are allowed here to label an edge by the
identity. After this operation some of the edges of $D$ are
'subdivided'. We subdivide the rest of the edges of $D$ so that
the labels of the new edges lie in the same factor as the old ones
and the product of their labels is equal to the label of the old
edge. We call this diagram van-Kampen diagram over the free
product.

We remark now that the $C'(\lambda )$ condition holds for this new
diagram, i.e. if $R_1,R_2$ are adjacent regions of the diagram
then
$$length (R_1\cap R_2)\leq \lambda \min (length (\bd R_1),length
(\bd R_2))$$

 \begin{Thm}\label{codim} Let $G$ be a finitely generated group with a finite free product
 $C'(1/6)$ presentation $<F|R'>$. Assume further that all $r\in
 R'$ are cyclically reduced words and $\|r\|$ has even length. Then $G$ has a codimension 1 subgroup.
 \end{Thm}

\begin{proof}
We construct a complex for $G$ as usual. If $F=*A_i$ and $K_i$ are
complexes with a single vertex $x_i$ such that
$\pi_1(K_i,x_i)=A_i$ we take the wedge product of the $K_i'$s
identifying all $x_i$'s. For each $r\in R'$ we glue a 2-cell to
$\vee K_i$ in the obvious way to obtain a complex $K$ such that
$\pi _1(K,x)=G$. We argue now in a way similar to Wise (\cite
{W}). We slightly change approach and we consider bouquets of
circles that go through $x$ rather than tracks. We explain now how
we construct a bouquet of circles $\Gamma $ which will give the
codimension 1 subgroup.

 Let
$r=a_1...a_{2n}$ be the normal form in $F$ of $r\in R'$. We
represent the 2-cell $c(r)$ corresponding to $r$ as a polygon
where the $a_i's$ are the labels of the sides of this polygon. The
bouquet of circles has a single vertex $x$ and a set of edges
corresponding to `diagonals' of these polygons. We fix $r\in R'$
as above we pick the diagonal joining the beginning of the
$a_1$-edge to the vertex opposite to it, i.e. the end of the
$a_{n}$ edge.

 We remark now that since $c(r)$ has an even number of
sides each vertex has a vertex opposite to it, so we associate to
this vertex the diagonal joining it with the opposite vertex. We
remark that any vertex is determined by the edges adjacent to it.
For example the beginning of the $a_1$ edge is the vertex
corresponding to the consecutive edges $a_{2n},a_1$. We consider
now the equivalence relation on vertices of the $c(r)'s$ generated
by the following relation: The vertex $b_i,b_{i+1}$ of $c(r_1)$ is
equivalent to the vertex $c_j,c_{j+1}$ of $c(r_2)$ if
$b_i,b_{i+1}$ lie in the same free factors as $c_j,c_{j+1}$. We
remark that $r_1$ might be equal to $r_2$ in this definition.

 Now for each $r$ we
consider all vertices equivalent to the vertices of the chosen
diagonal. We add to the bouquet of circles all diagonals
corresponding to these vertices and we call the graph obtained in
this way $\Gamma $. We remark that $\Gamma $ is a bouquet of
circles if we see it as an abstract graph but if we see it as
immersed in $K$ its edges are likely to intersect each other in
the middle point of the polygons.

 $\Gamma $ corresponds to a subgroup of $G$.
Indeed each diagonal gives a generator, for example the diagonal
joining $a_1,a_n$ gives the generator $a_1a_2...a_n$. Let's call
this subgroup $H$. We will show that $H$ is a codimension 1
subgroup of $G$.

\begin{Lem} There is a tree $\tilde \Gamma \subset \tilde K$ which
has as edges diagonals of 2-cells which is invariant under $H$.
$\tilde \Gamma $ separates $\tilde K$ in at least 2 components.
\end{Lem}
\begin{proof}
Let $v\in \tilde K$ be a vertex. We define now a connected graph
in $\tilde K$ as follows: We say that two vertices are related if
they are opposite. We take the equivalence relation generated by
this relation and we consider the equivalence class of $v$. Let
$\tilde \Gamma $ be the graph obtained by joining opposite
vertices in this equivalence class by diagonals. We claim that
$\tilde \Gamma $ is a tree. If it is not a tree then there is a
path $p$ in $\tilde \Gamma $ such that both endpoints of $p$ lie
on the same 2-cell of $R'$ and $p$ is not a single edge (a
diagonal). Let's say $a,b$ are the endpoints of $p$ and they lie
on a 2-cell $\sigma $. Let $q$ be a path on $\bd \sigma $ joining
$a,b$. We may assume $q$ to have minimal normal form length in the
free product among the 2 possible paths. Now $p\cup q$ is a closed
loop. We change now $p$ by replacing each diagonal by the
corresponding path on the boundary on which the diagonal lie. We
note that we have two choices and we replace the diagonals so that
the path we obtain by replacing all of them corresponds to a
reduced word of $F$. Let $p'$ be the path we obtain in this way.
We may arrange also that $p'\cup q$ is reduced at the vertex $a$
(unless $a=b$). We consider now the van-Kampen diagram over the
free product for $p'\cup q$ and we remark that if it has an
$i$-spur for $i\leq 2$ then the boundary of this $i$-spur contains
a neighborhood of the vertex $b$. But this contradicts Corollary
\ref{ladder}. It follows that $\tilde \Gamma $ is a tree. By
construction $\tilde \Gamma $ is invariant under $H$ and separates
locally (hence also globally) $\tilde K$.

\end{proof}

The first part of the next lemma follows also from work of Osin
(\cite{Os}, see also \cite{De}). We include a proof here for the
sake of completeness.
\begin{Lem} The vertex groups $A_i$ and $H$ embed quasi-isometrically in
$G$. $H$ is a codimension 1 subgroup of $G$.
\end{Lem}
\begin{proof}
Let $a$ be a geodesic word in the Cayley graph of $A_i$. We will
show that $a$ is a quasi-geodesic in $\tilde K$. Let $S$ be the
generating set of $G$ and let $|w|$ be the length of a word in
$S$. Let
$$M=\max \{|r|:r\in R' \}$$
We define a new length function $L$ for words of $S$:
$$L(w)=M\|w\|+|w|$$

It is clear that an $L$-geodesic is a quasi-geodesic.

Indeed let $p$ be a geodesic in the 1-skeleton of $\tilde K$ with
respect to the length function $L$ with the same endpoints as $a$.
We consider the van-Kampen diagram over the free product for
$a\cup p$. We may assume that $a\cap p$ is equal to the endpoints
of $a,p$ since along the intersection of $a,p$, $a$ is
quasi-geodesic.

We remark now that this diagram has at most 2 $i$-spurs (for
$i\leq 2$), corresponding to the endpoints of $p$ so by corollary
\ref{ladder} this diagram is a ladder. By considering now the
usual van-Kampen diagram for $a\cup p$ we have that $|a|\leq M|p|$
so $a$ is quasi-geodesic.

We prove now that $H$ is quasi-isometrically embedded. Since $H$
acts freely on $\tilde \Gamma $ it is enough to show that $\tilde
\Gamma $ is quasi-isometrically embedded. Let $p$ be a geodesic
path in $\tilde \Gamma $ joining two vertices $v,u$ of $\tilde
\Gamma $.

We change $p$ by replacing each diagonal by the corresponding path
on the boundary of the cell on which lies the diagonal. We pick
this path so that the word of $F$ corresponding to the path is
reduced. Let $p'$ be the path we obtain in this way. Let $q$ be an
$L$-geodesic path joining $v,u$. As before we may assume that
$p',q$ intersect only at their endpoints. Again by the definition
of $p',q$ if we consider the van-Kampen diagram over the free
product for $p'\cup q$ we remark that it has at most 2 $i$-spurs
for $i\leq 2$. Hence this diagram is a ladder. By considering the
usual van-Kampen diagram we have that $length (p')\leq M\,
length(q)$. Since $q$ is a quasi-geodesic we have that $p'$ is a
quasi-geodesic, so $H$ is quasi-isometrically embedded.

Finally we show that $H$ is a codimension 1 subgroup. It suffices
to show that $\tilde \Gamma $ separates $\tilde K$ in at least 2
components which are not contained in a finite neighborhood of
$\tilde \Gamma $. By its definition $\tilde \Gamma $ separates
locally $\tilde K$. Since $\tilde K$ is simply connected, $\tilde
\Gamma $ separates $\tilde K$.

We introduce now some useful terminology. Let $r\in R'$ and let
$c_1c_2...c_n$ be the normal form of $r$ in $F$. Recall that $r$
is cyclically reduced. Let $k$ be smallest such that
$$\|c_1...c_k\|\geq \frac {n}{6}$$ We say then that $c_1...c_k$
is a \textit{piece} of $r$. Similarly we define pieces of all
cyclic permutations of $c_1c_2...c_n$.

 Let $R $ be a 2-cell in $\tilde K$ which intersects $\tilde
\Gamma $ on an edge $e$. Let $v,u$ be the vertices of $e$.

Let $c_1c_2...c_n$ be the label of $R$ starting from $v$ and
written in free product normal form. Let $s$ be the vertex
corresponding to the endpoint of a piece $p_1=c_1...c_k$ of $R$
starting at $v$. We construct a path starting from $v$ and lying
in the same component of $\tilde X-\tilde \Gamma $ as $s$. The
path starts by $p_1$. At $s$ we continue $p_1$ by a piece $p_2$ of
another 2-cell $R_2$ corresponding to $r_2\in R'$. We pick $R_2\ne
R$ and so that $p_1p_2$ is reduced in $F$. We continue inductively
in the same way picking each time a new 2-cell and a piece so that
the word we obtain is reduced in $F$. Let $\beta=p_1...p_n$ be the
path we obtain after $n$ steps. If $s_n$ is the endpoint of
$p_1...p_n$ we claim that $d(s_n,\tilde \Gamma)\to \infty $ as
$n\to \infty$. Indeed let $q$ be a geodesic joining $s_n$ to a
closest vertex $t\in \tilde \Gamma $. We consider a geodesic
$\gamma$ in $\tilde \Gamma $ joining $v$ to $t$.

We distinguish now two cases. Assume first that $u$ does not lie
on $\gamma $. We change $\gamma $ by replacing each diagonal by
the corresponding path on the boundary on which the diagonal lie
to obtain a path $\gamma '$. We make these replacements so that
the word of $F$ corresponding to the new path is reduced and
$p_1^{-1}\gamma '$ corresponds also to a reduced word in $F$.
Clearly this is possible since we have two choices for replacing
each diagonal and the normal form of each starts from a different
free factor. We consider now the loop $$\beta \cup \gamma '\cup
q$$ Since $q$ is geodesic the van-Kampen diagram for free products
for this loop has at most 2 $i$-spurs with $i\leq 2$ which appear
around the endpoints of $q$. It follows that this diagram is a
ladder (see corollary \ref{ladder}) hence the lengths of $q$ and
$\beta \cup \gamma '$ are comparable so the length of $q$ goes to
infinity as $n\to \infty $.

We deal now with the second case, i.e. we assume that $u$ lies on
$\gamma $. We modify then $p_1...p_n$ as follows. We replace $p_1$
by the path $q_1$ on the boundary of $R$ joining $s$ to $u$. We
note that the new path $\beta '=q_1p_2...p_n$ might not be reduced
at the endpoint of $q_1$. We replace $\gamma $ by a path $\gamma
'$ in the 1-skeleton of $\tilde K$ as before so that
$q_1^{-1}\gamma '$ is reduced in the free product $F$. We remark
that the van-Kampen diagram over free products for the loop
$$\beta '\cup \gamma '\cup q$$ is a ladder in this case too hence
the length of $q$ goes to infinity as $n\to \infty $.

Similarly we construct we see that the component of $R-\tilde
\Gamma $ that does not contain $v$ is not contained in a finite
neighborhood of $\tilde \Gamma $. It follows that $H$ is
co-dimension 1.
\end{proof}
\end{proof}
\section {The example}
\begin{Thm} Given any $n>0$ there is a one-ended hyperbolic group
$G$ such that \begin{itemize}
\item $dim \bd G\geq n$
\item $\bd G$ is separated by a Cantor set.
\item $G$ does not split.
\end{itemize}
\end{Thm}

\begin{proof} Let $A$ be a torsion free 1-ended hyperbolic group with property $T$ and such that
$dim(\bd A)\geq n$ (eg a lattice in $Sp(n,1)$). Let's say
$A=<a_1,...,a_k>$. We may assume that $a_i^m\ne a_j^r$ for any
$i\ne j$ and $m,r>0$. We take now another copy of $A$. For
notational convenience we denote the second copy by $B$ and its
generators by $<b_1,...,b_k>$. We consider now the free product
$A*B$ and we define $G$ to be the small cancellation quotient of
$A*B$ given by the relations:
$$r_{i,j}=(a_ib_j)(a_ib_j^2)(a_ib_j^3)(a_ib_j^4)\ \ \ 1\leq i,j \leq k$$
By theorem \ref{codim} of the previous section $G$ has a free
codimension 1 subgroup $H$. As we showed in the proof of the
theorem $H$ is quasi-isometrically embedded so a Cantor set
separates $\bd G$. We show now that $G$ has property $FA$ (i.e. it
does not split). Clearly $G$ is not an $HNN$ extension since the
abelianization of $A$ is trivial, so the abelianization of $G$ is
trivial. We show now that $G$ does not split as an amalgamated
product. Let's say $G=X*_CY$. Without loss of generality we may
assume that $A\subset X$ and $B\subset gXg^{-1}$ or $B\subset
gYg^{-1}$. Let $g=x_1...x_n$ be the normal form of $G$ in the free
product decomposition. By replacing $A,B$ by conjugates we may
assume that either $g=1$ and $B\subset Y$ or $x_1\notin X$.
However we see then that the word
$$r_{i,j}=(a_ib_j)(a_ib_j^2)(a_ib_j^3)(a_ib_j^4) $$ is reduced in $X*_CY$
 unless $a_i$ or $b_j$ is in $C$.
As all $r_{i,j}$ are equal to the identity this implies that $A=C$
and $B$ contained in $Y$ or $B=C$ and $A$ contained in $X$ but in
both cases, the splitting would be trivial.

We claim finally that $G$ is hyperbolic. Indeed this follows by
lemma 4.4 of \cite{Os}, and \cite{De}. For the reader's
convenience we sketch a proof here using lemma \ref{isop}. It is
enough to show that $G$ satisfies a linear isoperimetric
inequality. Let $w$ be a word on the generators of $G$ and let $D$
be a reduced van-Kampen diagram for $G$. As we describe in section
3 one obtains from $D$ a new diagram, let's say $D_1$, which is
called the diagram for $w$ over the free product. Since $A,B$ are
hyperbolic they satisfy some isoperimetric inequality of the form
$$A(p)\leq K\l(p)$$
for any simple closed path $p$ in the Cayley graph of $A$ or of
$B$.

It follows that if $p$ is a simple closed path of $D$ such that
all edges of $p$ lie in $A$ (or in $B$) and if $v$ is the vertex
of $D_1$ that we obtain by collapsing $p$ to a point then
$$d_v=l(p)\geq \frac {1}{K}A(p)$$
where $d_v$ is the degree of $v$. It follows that
$$A(D)\leq A(D_1)+K\sum _{v\in D_1^0}d_v$$

From lemma \ref{isop} we have the following inequality for the
diagram $D_1$:
$$\frac {1}{3}\sum _{v\in D_1^{0}}\frac{d_v}{2}-\frac {2E^\circ
}{7}\leq V^\bullet +\frac {E^\bullet}{7}$$

We have
$$\sum _{v\in D_1^{0}}\frac{d_v}{2} \geq E^\circ
 \Rightarrow \frac {2E^\circ }{7} \leq \frac {2}{7}\sum _{v\in
D_1^{0}} \frac{d_v}{2}$$

so
$$\frac {1}{3}\sum _{v\in D_1^{0}}\frac{d_v}{2}-\frac {2E^\circ
}{7}\geq \frac {1}{42} \sum _{v\in D_1^{0}}\frac{d_v}{2}$$

We have also $l(\bd D_1)\leq l(\bd D)$ and $V^\bullet, E^\bullet
\leq l(\bd D)$. So from lemma \ref{isop} we have
$$A(D_1)\leq 6l(\bd D)$$

and

$$\sum _{v\in D_1^0}d_v \leq 42 V^\bullet +7E^\bullet $$
so
$$A(D)\leq (6+49K)l(\bd D)$$
In other words $G$ verifies a linear isoperimetric inequality, so
it is hyperbolic.

\end{proof}

\begin{Rem} The above example shows also that for any $n$
there is a finitely presented group $G$ with $asdim\,G>n$ which is
separated coarsely by a uniformly embedded set $H$ of $asdim\,H=1$
and which does not split. This answers a question in \cite{P}.
\end{Rem}

\end{document}
\bye